\documentclass{article}
\usepackage{amsmath}
\usepackage{amssymb}
\usepackage{latexsym}
\usepackage{amsthm}
\usepackage{mathrsfs}
\usepackage{wasysym}
\usepackage{fancyhdr}
\usepackage{amsfonts}
\usepackage{epsfig}
\usepackage{epstopdf}
\usepackage{color}

\newtheorem{thm}{Theorem}[section]
\newtheorem{corollary}[thm]{Corollary}
\newtheorem{lemma}[thm]{Lemma}

\newtheorem{prob}[thm]{Problem}

\newcommand{\tp}{\text{\upshape top}}

\begin{document}
\title{Connectivity and $W_v$-Paths in Polyhedral Maps on Surfaces}
\author{Michael D. Plummer\thanks{
Department of Mathematics,
Vanderbilt University,
Nashville, TN 37215, USA,
Email: \tt{michael.d.plummer@vanderbilt.edu}}\,,
Dong Ye\thanks{Corresponding author. Department of Mathematical Sciences,
Middle Tennessee State University,
Murfreesboro, TN 37132, USA,
Email: {\tt{dong.ye@mtsu.edu}}.
The author was partially supported by a grant from the Simons Foundation (No. 359516).}\;  and
Xiaoya Zha\thanks{Department of Mathematical Sciences,
Middle Tennessee State University,
Murfreesboro, TN 37132, USA, Email: {\tt{xiaoya.zha@mtsu.edu}}.
Research supported by NSA Grant H98230-1-02192.}}

\date{\it This paper is dedicated to the memory of Victor Klee.}
\maketitle

\begin{abstract}
The $W_v$-Path Conjecture due to Klee and Wolfe states that any two vertices of a simple polytope
can be joined by a path
that does not revisit any facet.
This is equivalent to the well-known Hirsch
Conjecture. Klee proved that the $W_v$-Path Conjecture is true
for all 3-polytopes (3-connected plane graphs), and  conjectured even more, namely that the $W_v$-Path Conjecture is true for all general cell complexes. This
general $W_v$-Path Conjecture was verified for polyhedral maps on the projective
plane and
the torus by Barnette, and on the Klein bottle by Pulapaka and Vince. 
Let $G$ be a graph polyhedrally embedded in a surface $\Sigma$, and $x, y$ be two vertices of $G$. In this paper, we show that if there are three internally disjoint $(x,y)$-paths which are homotopic to each other, then there
exists a $W_v$-path joining $x$ and $y$. For every surface $\Sigma$, define a function $f(\Sigma)$ such that if
for every graph polyhedrally embedded in $\Sigma$ and for a pair of
vertices $x$ and $y$ in $V(G)$, the local connectivity $\kappa_G(x,y) \ge f(\Sigma)$, then there exists a
$W_v$-path joining $x$ and $y$. We show that $f(\Sigma)=3$ if $\Sigma$ is the sphere, and for all other 
surfaces
$3-\tau(\Sigma)\le f(\Sigma)\le 9-4\chi(\Sigma)$, where $\chi(\Sigma)$ is the Euler characteristic of 
$\Sigma$, and
$\tau(\Sigma)=\chi(\Sigma)$ if $\chi(\Sigma)< -1$ and 0 otherwise.
Further, if $x$ and $y$ are not cofacial, we prove that $G$ has at least $\kappa_G(x,y)+4\chi(\Sigma)-8$ internally disjoint $W_v$-paths joining
$x$ and $y$. This bound is sharp for the sphere. Our results indicate that the $W_v$-path problem is related to
both the local connectivity $\kappa_G(x,y)$, and the number of different homotopy
classes of internally disjoint $(x,y)$-paths as well as the number of internally disjoint $(x,y)$-paths in each homotopy class.\medskip

 \medskip


\end{abstract}

\section{  Introduction}

The {\it $W_v$-Path Conjecture} (or {\it Non-revisiting Path Conjecture}), originally due to Klee and Wolfe (cf. \cite{K1}),
states that
any two vertices of a simple polytope $P$ can be joined by a path that does not revisit any facet of $P$.
(Such a non-revisiting path is also called a {\it $W_v$-path}.)
Klee further conjectured that the $W_v$-Path Conjecture is
true for general cell complexes \cite{K3}.
Larman \cite{L} showed that this general $W_v$-Path Conjecture is false for a very general type of 2-dimensional complex
and later Mani and Walkup \cite{MW} found a 3-sphere counterexample. 
The original $W_v$-Path Conjecture for boundary complexes of polytopes
is known to be equivalent to two other well-known conjectures, the {\it Hirsch Conjecture} and the
{\it Danzig $d$-step Conjecture}, involving higher dimensional polytopes which in turn are important
in the continuing search for a practical polynomial algorithm for the simplex method of linear programming.
For proofs of these equivalences we direct the reader to \cite{K1,K2} and \cite{KK}. The Dantzig $d$-step Conjecture was verified by
Klee and Walkup for all bounded polyhedra for $d\le 5$ \cite{KW}. In 2012, the Hirsch Conjecture was shown to be false
by Santos (cf. \cite{Sant}.)

The first positive result related to the general $W_v$-Path Conjecture was obtained by Klee \cite{K1} who showed
that every pair of vertices of a 3-connected plane graph $G$ (or ``3-polytope") are joined by a
$W_v$-path.  (See also \cite{Go} ad Gr\"unbaum \cite{Gr}.)
One of the nice properties of 3-connected plane graphs is that their faces meet ``properly''.
(Here and throughout the rest of the paper we consider a face to include its boundary.)
That is, they meet at a single vertex, a single edge or not at all.
This idea has been generalized to surfaces other than the plane by the notion of a {\it polyhedral embedding}.
An embedding of a graph $G$ in a surface $\Sigma$ is {\it polyhedral} if every face is a closed disk and any two faces of the embedding meet properly,
which is equivalently to saying the representativity (face-width)
of the embedding is at least 3 (cf. \cite{MT}).
It follows that a graph admitting a polyhedral embedding must be 3-connected (cf. \cite{MT}).

The general $W_v$-Path Conjecture has also been studied for polyhedral embeddings of graphs in general 2-dimensional surfaces as well.
(Here by ``2-dimensional surface'' we mean a connected compact 2-manifold without boundary.)
The $W_v$-Path Conjecture in this context states that for every surface (orientable or non-orientable) and every graph polyhedrally embedded therein, there is a $W_v$-path joining every pair of distinct vertices.
Klee's result on 3-connected plane graphs was later extended
to graphs polyhedrally embedded
in the projective plane \cite{B2} and torus \cite{B3} by Barnette,
and in the Klein bottle by Pulapaka and Vince \cite{PV2}.
It is now known, however, that the $W_v$-Path Conjecture is false for every orientable surface of genus $g\ge 2$ and for every non-orientable surface of genus $\overline{g}\ge 4$ (cf. \cite{PV1}).
Hence the sole unsettled case is the non-orientable surface with $\overline{g}=3$.
For a summary of these results, see \cite{B4, P, PV1, PV2}. 
As positive results for the $W_v$-Path Conjecture are rare, the departure point in
the present paper is an attempt to ascertain what conditions
suffice to make the $W_v$-Path Conjecture hold.

Let $G$ be a polyhedrally embedded graph in a surface $\Sigma$. Given two distinct vertices $x$ and $y$ in
a graph
$G$, they are {\it cofacial} if they belong to the boundary of a common face.
If the cofacial vertices
$x$ and $y$ are adjacent, then
there is exactly one $W_v$-path joining them (the single edge $xy$). If they are not adjacent,
there are exactly two $W_v$-paths joining them,
namely the two paths forming the boundary of the face.
In 
this paper, we will focus on the case in which 
$x$ and $y$ are {\it non}-cofacial.
The {\it local connectivity $\kappa_G (x,y)$} of two vertices $x$ and $y$ is defined to be the maximum number
of internally disjoint paths joining $x$ and $y$, where two paths joining $x$ and $y$ are {\em internally disjoint}
if they have only $x$ and $y$ in common. A graph $G$ is {\it $k$-connected} if 
$\kappa_G(x,y)\ge k$ for any two vertices $x$ and $y$. 
We observe that the $W_v$-path problem is closely related to both the local connectivity $\kappa_G(x,y)$, and
the number of homotopy classes of $(x,y)$-paths as well as the number of $(x,y)$-paths in each homotopy class.
In order to describe our results, define $f(\Sigma)$ to be the
smallest value such that for every graph $G$ polyhedrally embedded in the surface $\Sigma$  
and for a  pair of vertices $x$ and $y$ of $G$, if $\kappa_G(x,y)\ge f(\Sigma)$,
then there exists a non-revisiting $(x,y)$-path. The following is one of our main results, in which $\chi (\Sigma)$ denotes the Euler characteristic of the surface $\Sigma$.

\begin{thm} \label{thm:main-general}
Let $\Sigma$ be a closed surface. Then $f(\Sigma)=3$ if $\Sigma$ is the sphere. For all other surfaces $3-\tau(\Sigma)\le f(\Sigma)\le 9-4\chi(\Sigma)$, where $\tau(\Sigma)=\chi(\Sigma)$ if $\chi(\Sigma)<-1$ and $0$ otherwise.
\end{thm}

The lower bound is obtained by construction.
In order to verify the upper bound,
we introduce the concept of dual curve
for 
 revisits, which turns out to be very useful
in bounding the number of revisits and the number of
different homotopy classes of $(x,y)$-paths. In particular,
we prove the following result.

\begin{thm}\label{thm:maintoo}
Let $G$ be a graph polyhedrally embedded in a surface $\Sigma$, and $x$ and $y$ be two non-cofacial vertices. If there exist three internally disjoint $(x,y)$-paths which are homotopic to each other, then there exists a non-revisiting $(x,y)$-path.
\end{thm}

The above result says three internally disjoint homotopic $(x,y)$-paths implies the existance of one non-revisiting path.
However, Theorem~\ref{thm:main-general} indicates that, for each surface $\Sigma$ with Euler characteristic $\chi(\Sigma)<-1$, a graph $G$ polyhedrally embedded in $\Sigma$ may not have $W_v$-path between two vertices $x$ and $y$ if  there are less than $3-\chi(\Sigma)$ paths joining them. 
This shows that the non-revisiting path problem is related to the homotopy classes of $(x,y)$-paths. 

Another application of our method
provides a very short proof for the upper bound for 
the face touching number
of a 3-connected graph embedded in a surface, which was originally proved by Sanders \cite{Sand}
using a discharging argument.

Besides the existence of $W_v$-paths, Barnette \cite{B1} also generalized the $W_v$-path result for the plane in a different direction.
He proved that if two vertices of a graph polyhedrally embedded in the plane are non-cofacial, then they are joined by at least three 
internally disjoint $W_v$-paths.
  Richter and Vitray \cite{RV} proved that, in fact, if a graph is embedded in any surface with representativity at least 4,
there are at least two internally disjoint homotopic
$W_v$-paths joining any two non-cofacial vertices.
In this paper, we also derive the following new relationship between the number of internally
disjoint $W_v$-paths joining two vertices $x$ and $y$ and the local connectivity
$\kappa_G(x,y)$.

\begin{thm}\label{thm:main2}
Let $G$ be a graph polyhedrally embedded in a surface $\Sigma$, and $x$ and $y$ be two non-cofacial vertices. Then  $G$ has at least $\kappa_G(x,y)+4\chi(\Sigma)-8$ internally disjoint non-revisiting $(x,y)$-paths.
\end{thm}

The bound in Theorem~\ref{thm:main2} is sharp for the sphere.
This will be proved in Section 3. An
even 
better bound for the projective plane will be given in Section 4. 
\bigskip

\section{Dual curves and contractible revisits}

We begin with some definitions and notation.
Let $G$ be a graph polyhedrally embedded in a surface $\Sigma$ and let $x$ and $y$ be two vertices of $G$.
Let $P$ be a path joining $x$ and $y$.
A face $F$ is {\it revisited} by the path $P$ if $F\cap P$ has at least two components.
Let $c( F\cap P)$ to be the number of components of $F\cap P$.
The {\it total revisit number} of $P$ is $r_P=\sum_{F}(c(F\cap P)-1)$.
Let $S_1, S_2,...,S_k$ be the connected components of $F\cap P$.  Throughout the paper, $S_i$ for some integer $i$ always 
stands for a connected component of the intersection of an $(x,y)$-path and some face.
A pair $\{S_i, S_j\}$ is called a {\it revisit} to $F$ by $P$.
(See Figure~\ref{fig:non-contractible} (left)
where $P$ is represented by the thick edges joining $x$ and $y$ and $F$ is the face exterior to the outside octagon.)

\begin{figure}[!hbtp] \refstepcounter{figure}\label{fig:non-contractible}
\begin{center}
\includegraphics[scale=1.25]{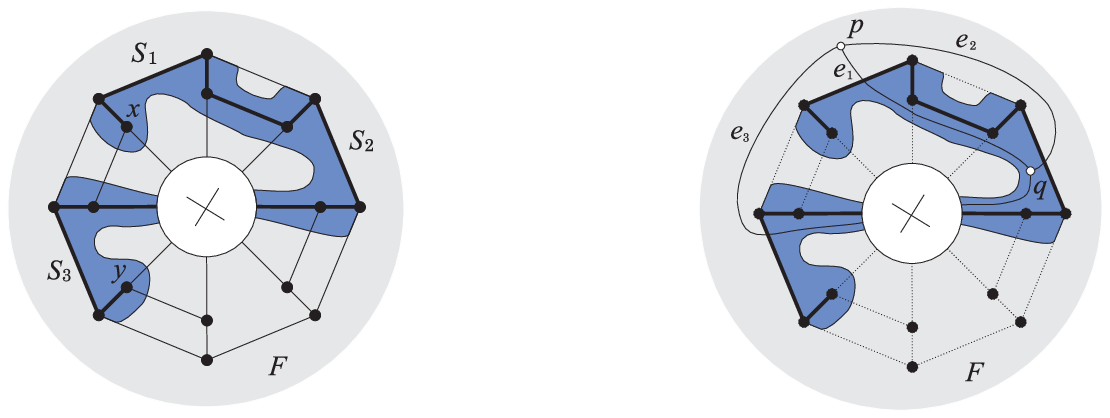}\\
{Figure~\ref{fig:non-contractible}: Dual curves, contractible and non-contractible revisits.}
\end{center}
\end{figure}

For an $(x,y)$-path $P$ joining $x$ and $y$, let $N_{\epsilon} (P)$ be a closed $\epsilon$-neighborhood
of $P$ such that $N_{\epsilon} (P)$ is homotopic to a disk. 
Denote the interior of $F$ by $Int (F)$.
Let $N_{\tp}(P)=N_{\epsilon} (P)\setminus Int (F)$.
Then $N_{\tp}(P)$ is homeomorphic to a closed disk.
(See the dark grey area in Figure~\ref{fig:non-contractible}.)

Let $p$ and $q$ be two points lying in the interiors of $F$ and $N_{\tp}(P)$ respectively.
We construct an auxiliary graph
$H(p,q)$
also
embedded in $\Sigma$ such that $H(p,q)$ has
two vertices $p$ and $q$ and $k$ edges $e_i, 1\le i\le k$,
where $e_i$ joins $p$ and $q$ through $S_i$.
Let $\ell_{ij}=e_i\cup e_j$.
Then $\ell_{ij}$ is a simple closed curve which we will call the {\it dual closed curve} of
the revisit $\{S_i,S_j\}$.
In general, of course, a dual closed curve can be contractible or non-contractible.
A revisit $\{S_i, S_j\}$ is {\it non-contractible} if $\ell_{ij}$ is non-contractible, and {\it contractible}, otherwise. 
For example, see Figure~\ref{fig:non-contractible}
(right) in which
$\ell_{13}$ and $\ell_{23}$ are non-contractible
and $\ell_{12}$
is contractible. 
Therefore $\{S_1,S_2\}$ is
a contractible revisit,
but $\{S_1, S_3\}$
is a non-contractible revisit in Figure~\ref{fig:non-contractible} (left).
Additional definitions and notation will be introduced below as needed.

\begin{lemma}\label{lem:non-contractible}
Let $G$ be a 3-connected graph embedded in a surface $\Sigma$ and let $x$ and $y$ be two vertices of $G$.
Let $P$ be an $(x,y)$-path revisiting a face $F$ such that every revisit is non-contractible.
Then the number of components of $P\cap F$ is at most $4-2\chi(\Sigma)$.
\end{lemma}

\begin{proof}  Assume that $P\cap F=\{S_1,\ldots,S_k\}$ where $k\ge 2$.
Let $H=H(p,q)$ be the auxiliary graph defined as above, namely let $p$ and $q$ be two vertices belonging to 
the interiors of
$F$ and $N_{\tp}(P)$ respectively,
and suppose, for $i=1\ldots k$,
$e_i$ is an edge joining $p$ and $q$ 
through $S_i$.  Then the graph $H$ is embedded in $\Sigma$.
Since every revisit $\{S_i, S_j\}$ is non-contractible, every dual closed curve $\ell_{ij}=e_i\cup e_j$ is non-contractible.  
If $e_i\cup e_j$ bounds a face of $H$, then the interior of the face is not homemorphic to an open disk. In this case, cut the face along $e_i\cup e_j$ and cap off its boundary curve $e_i\cup e_j$. For every face of size 2, apply this operation so that eventually we generate a new surface $\Sigma'$ such that every face of size 2 of $H$ embedded in $\Sigma'$ has its interior homemorphic to an open disk.
Let $f_2$ be the number of faces of $H$ with size 2. Then $\chi(\Sigma')\ge \chi(\Sigma)+f_2$. 
Consider the graph $H$ embedded in the surface $\Sigma'$. Then $H$ has 
$f_2$ faces of size 2 which are closed disks. Let $F(H)$ be the set of all faces of $H$. Then by Euler's formula,
\[2-|E(H)|+|F(H)|\ge \chi(\Sigma')\] where equality holds if the interior of every face is homemorphic to an open disk. 
Let $f_{4^+}$ be the number of faces of $H$ with size at least 4, then $|F(H)|=f_2+f_{4^+}$.   Hence
\[ 2-k+ (f_2+  f_{4^+}) \ge \chi(\Sigma)+f_2.\]
It follows that $\chi(\Sigma)\le 2-k+f_{ 4^+}$. 
Note that $2k=2|E(H)|\ge 2f_2+4f_{4^+}\ge 4f_{4^+}$, and further, $f_{4^+}\le k/2$. 
Combining this inequality with $\chi(\Sigma)\le 2-k+f_{4^+}$,
it then follows that $k\le 4-2\chi(\Sigma)$.
\end{proof}

The {\it face touching number} of two faces $F_1$ and $F_2$ is the number
of components of $F_1 \cap F_2$.
The face touching number of a graph $G$
is the maximum face touching number
over all pairs of faces of $G$. Assume that $F_1\cap F_2=\{S_1,\ldots, S_k\}$ where $S_i$ is a connected component of $F_1\cap F_2$.
If $G$ is 3-connected, every component $S_i$
is a single edge
or vertex
and hence the boundary of $F_1$ contains at least one edge
$xy$ which is not on the boundary of
$F_2$.
Then deleting $xy$ from the boundary of $F_1$ results in a path, which we will denote
by $P$.
Note that $P\cap F_2=F_1\cap F_2$ as $xy\notin F_1\cap F_2$.
By the 3-connectivity of $G$, we can conclude that every revisit $\{S_i,S_j\}$ to $F_2$ by $P$
is non-contractible.
Otherwise, $\{S_i, S_j\}$
contains a
2-vertex-cut of $G$ as
the dual closed curve $\ell_{ij}$ of $\{S_i, S_j\}$ is contractable
and hence separating, a contradiction to the 3-connectivity of $G$. By Lemma~\ref{lem:non-contractible},
we have the following result on face touching numbers of 3-connected graphs,
which was originally proved by Sanders \cite{Sand}
using a discharging argument.

\begin{corollary}[\cite{Sand}]
Let $G$ be a 3-connected graph embedded in a surface $\Sigma$.
Then the face touching number of $G$ is at most $4-2\chi(\Sigma)$.
\end{corollary}

\noindent{\bf Remark.} The bound of Lemma~\ref{lem:non-contractible} is tight in that equality may hold.
Sanders
constructed examples to illustrate that the face touching number of a 3-connected graph
can reach $4-2\chi(\Sigma)$.
Again,
if
one traverses
a path $P$ from the boundary of one of the two faces
of the examples of Sanders, then $P$ revisits the other face $4-2\chi(\Sigma)$ times.

\bigskip

The following lemma gives a condition under which the number of contractible revisits can be reduced.

\begin{lemma} \label{lem:contractible}
Let $G$ be a graph polyhedrally embedded in a surface $\Sigma$, and $x$ and $y$
be two non-cofacial vertices.  Suppose
${\cal P}=\{P_1,\ldots,P_k\}$ $(k\ge 3)$ is a set of $k$ internally disjoint $(x,y)$-paths.
If a face $F$ has a contractible revisit
by path $P_i$, there exists a path $P'_i$ such that
${\cal P'}=(\mathcal P\backslash \{P_i\}) \cup \{P_i'\}$ is a set of $k$ internally disjoint $(x,y)$-paths with
$r_{\cal P'} < r_{\cal P}$.
\end{lemma}

\begin{proof}
Assume that $F\cap P_i=\{S_1, S_2,\cdots, S_t\}$. Without loss of generality, we may assume that $\{S_1, S_2\}$ is a contractible revisit. Then the
dual curve $\ell_{12}$ bounds a disk $D$. If the disk $D$ contains any
other component of $F\cap P_i$, say $S_j$, then
the dual curve $\ell_{1j}$ bounds another disk $D'$ which is contained inside $D$. Since the number of revisits is finite,
there exists a revisit such that 
its dual curve 
bounds a disk which does not contain any other revisits. 
Therefore,
 without loss of generality, assume that
$D$ does not contain any other component of $F\cap P_i$.

Let $v_1, v_2$ be two endvertices of $S_1$ and $u_1, u_2$ be two endvertices of $S_2$ such that $v_1,v_2, u_1$ and $u_2$ appear on the boundary of $F$ in clockwise order. Note that, it is possible that $v_1=v_2$ and/or $u_1=u_2$.
Assume that the segment of the boundary of $F$ inside the disk $D$ from $v_2$ to $u_1$ is denoted by 
$v_2Fu_1$. 

If one of $x$ and $y$ is outside the disk $D$ and the other is inside the disk $D$, then an $(x,y)$-path $P_j\in \mathcal P$ with $j\ne i$ will intersect the boundary of disk $D$ by the Jordan Curve Theorem; in other words, $P_j$ intersects $P_i$, contradicting
the fact
that $P_i$ and $P_j$ are internally disjoint.
Hence 
$x$ and $y$ are either 
both
outside the disk $D$ or
both
inside the disk $D$.

First, assume that both $x$ and $y$ are outside the disk $D$.
Then every $P_j\in \mathcal P$ with $j\ne i$ is disjoint from $v_2Fu_1$.
Further, assume that $P_i$ 
is traversed 
from $x$ to $v_1$ first and then $v_2$. 
By the definition of dual curve, the segment of path $P_i$ from $S_1$ to $S_2$ together with $v_2Fu_1$ 
forms a curve homotopic to $\ell_{12}$. 
Since $y$ is outside 
$D$, it follows that $P_i$ passes through $u_1$ first and then $u_2$.
Let $v_2P_iu_1$ stand for the subpath of $P_i$ joining $v_2$ and $u_1$, and let $P_i'=(P_i\backslash v_2P_iu_1)\cup v_2Fu_1$.
Then $P_i'$ is internally disjoint from $P_j\in \mathcal P$ with $j\ne i$.
Since $y$ is outside of the disk $D$,  the segment of $P_i\backslash v_2P_iu_1$ from $u_2$ to $y$ does not intersect the cycle $v_2Fu_1\cup v_2P_iu_1$.
Note that every face visited by $v_1Fu_1$, except $F$, 
lies inside 
the disk bounded by $v_2Fu_1\cup v_2P_iu_1$.
It
then
follows that $v_1Fu_1$ does not revisit any other face $F'$, for if there were such a revisit, the two faces $F$ and $F'$ would touch twice,  
contradicting the fact
that $G$ is polyhedrally embedded in $\Sigma$. So $r_{{\cal P}'_{1}}<r_{\cal P}$.

So in the following, assume that both $x$ and $y$ are inside 
the disk $D$. 
Then all other $(x,y)$-paths $P_j\in \mathcal P$ with $j\ne i$ are inside 
$D$, for otherwise, $P_j$ intersects $P_i$, a contradiction
of the fact that
$P_i$ and $P_j$ are internally disjoint. 
Now let $P_i'=(P_i\backslash v_1P_iu_2)\cup u_2Fv_1$.
Then $P_i'$ is disjoint from $P_j$ since $u_2Fv_1$ is outside
$D$.
Any face $F'$ visited by $u_2Fv_1$ is outside
$D$ and is not visited by
the
segments $P_i\backslash v_1P_iu_2$ which are inside 
$D$.
So $u_1Fv_1$ does not revisit any face of $G$ since $G$ is polyhedrally embedded in $\Sigma$. 
Therefore  $r_{{\cal P}'_{1}}<r_{\cal P}$. This completes the proof. 
\end{proof}

If $G$ is polyhedrally embedded in the plane, then every revisit to a face by a path is contractible.
So the following result,
which strengthens a classical result of Barnette on $W_v$-paths (\cite{B1}), is an immediate corollary of
Lemma~\ref{lem:contractible}.

\begin{thm}\label{thm:plane}
Let $G$ be a graph polyhedrally embedded in the sphere and $x,y$ two non-cofacial vertices of $G$.
Then there are at least $\kappa_G(x,y)$ internally
disjoint $W_v$-paths joining $x$ and $y$.
\end{thm}

\section{Polyhedral maps on general surfaces}

In this section, we will prove our main results, namely Theorems 1.1, 1.2 and 1.3.

Let $x$ and $y$ be two vertices of a graph $G$ polyhedrally embedded in a surface $\Sigma$.
Two internally disjoint $(x,y)$-paths $P$ and $P'$ are {\em homotopic} to each other if
$P\cup P'$ bounds an open disk of $\Sigma$. Given a family $\mathcal P$ of internally disjoint $(x,y)$-paths, a {\em homotopy class} $\mathcal P'$ of $\mathcal P$ is a subfamily of $\mathcal P$ such that any two paths
of $\mathcal P'$ are homotopic to each other and any path $P\in \mathcal P\backslash \mathcal P'$ is
not homotopic to any path in $\mathcal P'$.
Note that, if $\Sigma$ is the sphere,
then all internally disjoint $(x,y)$-paths are homotopic to each other and hence
there is exactly one homotopy class of any given family of internally disjoint $(x,y)$-paths in this case.

\begin{lemma} \label{lem:class}
Let $G$ be a connected graph embedded in a surface $\Sigma$ different from the sphere, and suppose $x,y\in V(G)$.
Then the number of homotopy classes of a family of internally disjoint $(x,y)$-paths is no more than $4-2\chi(\Sigma)$.
\end{lemma}

\begin{proof}  Let $\mathcal P$ be a family of internally disjoint $(x,y)$-paths and let $k$ be the total number of homotopy classes of $\mathcal P$.
Since $\Sigma$ is not the sphere, $\chi(\Sigma)< 2$.
If $k=1$, then the lemma holds trivially.
So in 
the following, suppose that $k\ge 2$.
Choose one $(x,y)$-path $P_i$ ($i=1,2,...k$) from each homotopy class.
Then no two of $P_1,\ldots, P_k$
are homotopic to each other.

We construct an auxiliary graph $H$ embedded in $\Sigma$ as follows:
let $V(H)=\{x,y\}$ and $E(H)=\{e_1,\ldots,e_k\}$ where $e_i$ is a single edge joining $x$ and $y$
and is homotopic to path $P_i$.
Then $H$ is a bipartite multigraph with two vertices and $k$
edges.
Since $P_i$ is not homotopic to $P_j$ for $j\ne i$, the same is true for $e_i$ and $e_j$.
Therefore, $e_i\cup e_j$ is a non-contractible cycle of $H$. If $e_i\cup e_j$ bounds a face, then the interior of the face is not homemorphic to an open disk. 
An argument similar to that used in the proof of Lemma~\ref{lem:non-contractible} shows that $k=|E(H)|\le 4-2\chi(\Sigma)$.
\end{proof}

\noindent
The following result 
illustrates
an important connection between
homotopy classes and $W_v$-paths.

\begin{lemma}\label{lem:homo-Wv}
Let $G$ be a graph polyhedrally embedded in a surface $\Sigma$, and let $x$ and $y$
be
two
non-cofacial vertices of $G$. Let $\mathcal P$ be
a homotopy class of a family of internally disjoint $(x,y)$-paths of $G$, and assume that $D$
is the minimal disk containing all paths $\mathcal P$. Then $D$ contains at least $|\mathcal P|-2$
internally disjoint non-revisiting $(x,y)$-paths.
\end{lemma}
\begin{proof}
Assume that $\mathcal P=\{P_1,\ldots, P_k\}$ is a homotopy class of a family of internally disjoint $(x,y)$-paths.
If
$k\le 2$, the lemma holds trivially. So suppose that $k\ge 3$. As $D$ is the minimal disk
containing all paths in $\mathcal P$, we can conclude that $D$ is bounded by two paths
in $\mathcal P$, say $P_1$ and $P_k$.

Note that all $(x,y)$-paths contained in $D$ are homotopic.
We choose a set of
$k$ internally disjoint $(x,y)$-paths in $D$, denoted by $\mathcal P'=\{P_1',\ldots, P_k'\}$,
such that the total revisit number of $\mathcal P'$ is minimal.
Relabeling if necessary, we may assume
that all
paths in $\mathcal P'$ are contained in a disk bounded by $P_1'$ and $P_k'$.
Every revisit to a face $F$
in $D$ by an $(x,y)$-path in $\mathcal P'$ is contractible.
By Lemma~\ref{lem:contractible}
and the choice of $\mathcal P'$, all paths in $\mathcal P'$ except $P_1'$ and $P_k'$, are $W_v$-paths joining
$x$ and $y$.
It then follows immediately that $D$ contains at least $|\mathcal P|-2$ non-revisiting $(x,y)$-paths.
\end{proof}

Theorem~\ref{thm:maintoo} follows immediately from Lemma~\ref{lem:homo-Wv}. Now, we are going to prove Theorem~\ref{thm:main2}.\bigskip

\noindent{\it Proof of Theorem~\ref{thm:main2}.}
Assume that $\kappa_G(x,y)=k$. Then $G$ has $k$ internally disjoint $(x,y)$-paths. Assume
these $k$ internally disjoint $(x,y)$-paths can be
partitioned into $t$ homotopy classes
$\mathcal P_1, \ldots, \mathcal P_t$.
It follows from Lemma~\ref{lem:class},
that $t\le 4-2\chi(\Sigma)$. Let $D_i$ be the minimal disk containing all paths in
$\mathcal P_i$. By Lemma~\ref{lem:homo-Wv}, each disk $D_i$
contains at least $|\mathcal P_i|-2$ internally disjoint
non-revisiting $(x,y)$-paths.
Therefore,
the total number of internally
disjoint non-revisiting $(x,y)$-paths is at least
\[\sum\limits_{i=1}^t (|\mathcal P_i|-2)=\sum_{i=1}^t |\mathcal P_i| -2t\ge k+4\chi(\Sigma)-8.\]
\qed \medskip

Theorem~\ref{thm:main2}
guarantees that if the local connectivity $\kappa_G(x,y)$ is large enough, then $G$
has a non-revisiting  $(x,y)$-path.
It
would be interesting to find the
{\it minimum} local connectivity for graphs polyhedrally embedded
in a surface $\Sigma$ which guarantees
the existence of a non-revisiting $(x,y)$-path.
Define $f(\Sigma)$ to be the smallest number $k$ such that for any graph $G$ polyhedrally embedded in $\Sigma$
and any two vertices $x$ and $y$ in $G$,
if $\kappa_G(x,y)\ge k$, then $G$ has at least one non-revisiting $(x,y)$-path.
By Theorem~\ref{thm:plane}, $f(\Sigma)=3$ if $\Sigma$ is the sphere. 
The
results obtained in \cite{K1, B1, B2, B3, PV2} show that $f(\Sigma)=3$ for surfaces with $\chi(\Sigma)\ge 0$. 
For all other surfaces $\Sigma$, the exact value of $f(\Sigma)$ remains unknown.
By Theorem~\ref{thm:main2}, we have the following result which provides an upper bound
of $f(\Sigma)$ for surfaces $\Sigma$ with $\chi(\Sigma)<0$.   

\begin{corollary} \label{cor:main}
Let $\Sigma$ be a surface. Then $f(\Sigma)=3$ if $\chi(\Sigma)\ge 0$ and $f(\Sigma)\le 9-4\chi(\Sigma)$ otherwise.
\end{corollary}

In the following, we are going
to construct examples 
in which the lower bound
given in Theorem~\ref{thm:main-general} holds.
Since $G$ is polyhedrally embedded in a surface $\Sigma$, it follows that $G$ is 3-connected and
therefore $f(\Sigma)\ge 3$. The lower bound
in Theorem~\ref{thm:main-general} holds trivially for surfaces $\Sigma$ with $\chi(\Sigma)\ge -1$. 
Now, we construct examples to illustrate the
even
better lower bound $f(\Sigma)\ge  3-\chi(\Sigma)$
for surfaces with $\chi(\Sigma)<-1$, i.e., all orientable surfaces with genus $g\ge 2$ and non-orientable 
surfaces with genus $\overline g\ge 4$.
 
 \medskip

For each orientable genus $g\ge 2$ and each non-orientable genus $\overline{g}\ge 4$,
we now exhibit examples of graphs with these genera
having
the property that they contain vertices $x$ and $y$
such that
$\kappa_G(x,y)=2g =2-\chi(\Sigma)$ in the orientable case
and
such that
$\kappa_G(x,y)=\overline{g}=2-\chi(\Sigma)$ in the non-orientable case,
but there is no $W_v$-path joining $x$ and $y$.\medskip

\begin{figure}[!hbtp] \refstepcounter{figure}\label{fig:orient}
\begin{center}
\includegraphics[scale=1.2]{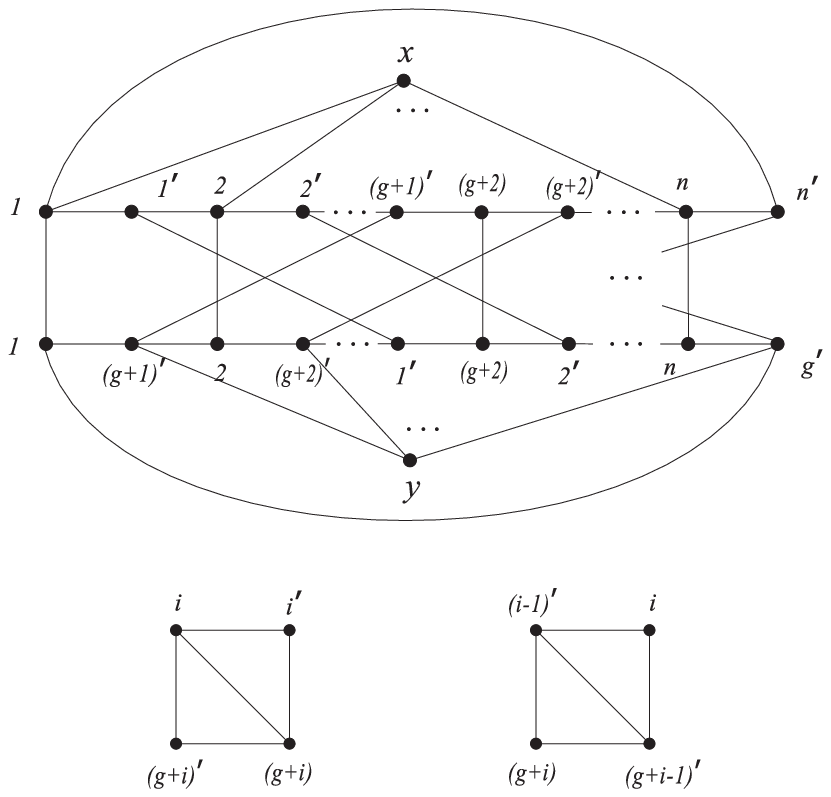}\\
{Figure~\ref{fig:orient}: Orientable surfaces.}
\end{center}
\end{figure}

First we consider the orientable case for all $g\ge 2$. Let $x$ and $y$ be two distinct vertices.
Let $C_x=11'22'33'\cdots nn'1$ be a cycle on $2n$ vertices such that $N(x)=\{1,2,\ldots,n\}$, where $n=2g$.
Similarly,
let $C_y=1(g+1)'2(g+2)'\cdots gn'(g+1)1'\cdots ng'1$ be a second
$(2n)$-cycle such that $N(y)=\{1',2',\ldots,n'\}$, again where $n=2g$.
For $i=1,\ldots,n$, join vertex $i$ in $C_x$ and $i\in C_y$ as well as
$i'$ in $C_x$ and $i'\in C_y$.
Call the resulting graph $H_1$. (See Figure~\ref{fig:orient}.) The cyclic orders of
edges incident with vertices as shown
in Figure~\ref{fig:orient} define a {\it rotation} scheme which represents an embedding  of $H_1$ in an orientable surface $\Sigma$.
By Euler's formula, the surface $\Sigma$ has genus $g$. The 
faces of the embedding derived from the rotation system shown in Figure~\ref{fig:orient}
are of the form $xii'(i+1)x$, $yi'(g+i+1)(i+1)'y$,  and  $ii'i'(g+i)(g+i)(g+i)'(g+i)'i$ where
all integers are taken modulo $n$.

Now envision the graph $H_1$ embedded in 
this
surface $\Sigma$.
Next contract all edges of the form $ii$ and $i'i'$.
Call the resulting graph $H_2$. Then $H_2$ inherits the embedding of $H_1$  in the surface $\Sigma$
such that each facial 8-cycle in $H_1$ of the form
$ii'i'(g+i)(g+i)(g+i)'(g+i)'i$ in $H_1$ corresponds to a facial 4-cycle
$ii'(g+i)(g+i)'i$ in $H_2$, and other facial 4-cycles of $H_1$ are still facial 4-cycles of $H_2$.
Now the graph $H_2$ is embedded in the surface $\Sigma$ where every face is bounded by a 4-cycle.
This embedding is not polyhedral because, for example, the 4-faces $ii'(g+i)(g+i)'i$ and $(i-1)'i(g+i-1)'(g+i)(i-1)'$ share vertices $i$ and $(g+i)$ which are two components of the intersection of the face boundaries. So we add some additional diagonal edges to some of these paired 4-cycles as follows:
for each $i =1,\ldots,n$, to the cycle $ii'(g+i)(g+i)'i$ we add the diagonal edge $i(g+i)$ and to the cycle $(i-1)'i(g+i-1)'(g+i)(i-1)'$ we
add the diagonal edge $(i-1)'(g+i-1)'$.
(See Figure 2.)

The resulting graph on $4g+2$ vertices, which we will call
$\Gamma_g$, is then polyhedrally embedded in the orientable surface $\Sigma$
of genus $g$ and $\kappa_G(x,y)=2g$.
Note that an $(x,y)$-path starting with an edge $xi$ revisits either a face incident with $x$ (for example $x(g+i)(g+i)'(g+i+1)x$ or $x(g+i)(g+i-1)'(g+i-1)x$) or a face  incident with $y$ (for example $i(g+i)'y(g+i-1)'i$). So 
 in
$\Gamma_g$ there are no $W_v$-paths joining $x$ and $y$. \medskip

\begin{figure}[!hbtp] \refstepcounter{figure}\label{fig:non-ori}
\begin{center}
\includegraphics[scale=1.2]{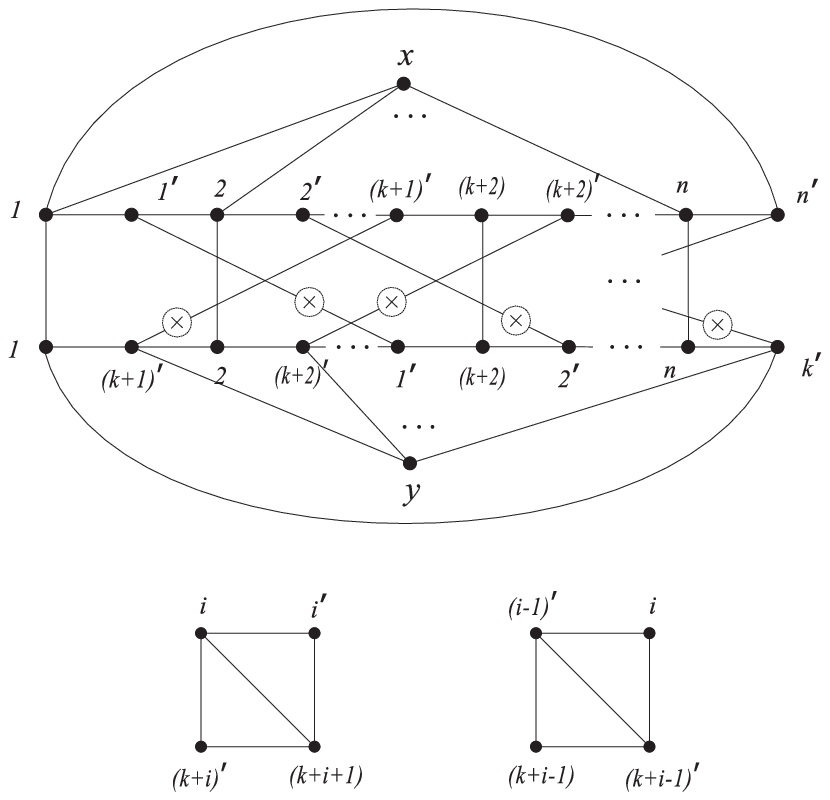}\\
{Figure~\ref{fig:non-ori}: Non-orientable surfaces.}
\end{center}
\end{figure}

We now turn to the non-orientable case. For the non-orientable surface $\Sigma$ where $\chi(\Sigma)=2-\overline g$ is {\it even} (i.e., $\overline g=2k$), we proceed as follows.\par

Let $x$ and $y$ be distinct vertices, and $C_x=11'22'\cdots nn'1$ be a $2n$-cycle with
$N(x)=\{1,2,\ldots,n\}$
and let
$C_y=1(k+1)'2(k+2)'\cdots kn'(k+1)1'\cdots n k'1$
be a second $2n$-cycle with
$N(y)=\{1',2',\ldots,n'\}$ where $n=\overline g$.
Join vertex $i$ of $C_x$ to $i$ of $C_y$ and vertex $i'$ of $C_x$ to $i'$ of $C_y$.
As in the orientable case, we also add all ``vertical'' edges of the form $ii$ and $i'i'$ and call the resulting graph $H_1$. 
This time, however, we position a separate crosscap on each the edges $1'1', 2'2',\ldots, n'n'$ in $H_1$
to obtain a non-orientable graph $\overline{H_1}$.  The rotation scheme as shown in
Figure~\ref{fig:non-ori} represents an embedding of $H_1$ in a non-orientable surface $\Sigma$. 
Again by Euler's 
formula, the surface $\Sigma$ has non-orientable genus $\bar g$.

We contract all edges of the form $ii$ and $i'i'$.
We denote by 
$\overline{H_2}$ the resulting graph embedded in the surface $\Sigma$.
In so doing, the 8-faces of the form
$ii'i'(k+1+i)(k+1+i)(k+i)'(k+i)'i$ and
$i'(i+1)(i+1)(k+i)'(k+i)'(k+i)(k+i)i'$
contract to the 4-faces
$ii'(k+1+i)(k+i)'i$ and $i'(i+1)(k+i)'(k+i)i'$
respectively.
As before, we obtain pairs of quadrilaterals which share two vertices on their boundaries which are not consecutive on either boundary.
So again we add the diagonal edges $i(k+i+1)$ and $(i-1)'(k+i-1)'$ to $\overline{H_2}$ to obtain a polyhedrally embedded graph which we shall call
$\Gamma_{\overline g}$.
In this embedded graph 
$\Gamma_{\overline g}$,
$\kappa_{\overline G}(x,y)=\overline g$. Again, in $\Gamma_{\overline g}$, an $(x,y)$-path
revisits either a face incident with $x$ or a face incident with $y$. Therefore, there is no $W_v$-path joining $x$ and $y$ 
in $\Gamma_{\overline g}$. \par

\begin{figure}[!hbtp] \refstepcounter{figure}\label{fig:cross-cap}
\begin{center}
\includegraphics[scale=1.3]{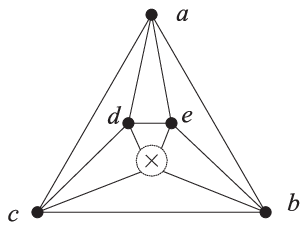}\\
{Figure~\ref{fig:cross-cap}: The added crosscap.}
\end{center}
\end{figure}

We can modify
the above construction for even non-orientable genera in order to treat the case when the non-orientable genus is {\it odd}
as follows.
Begin with the embedded graph $\overline{H_2}$ of even non-orientable genus $\overline g$ and select any triangular face $F$.
Denote it by  $F=abca$.
Now add two new adjacent vertices $d$ and $e$ and a new crosscap to the interior of $F$.
Join $a$ to $d$ and $e$, $b$ to $e$ and $c$ to $d$.
Finally, join $c$ to $e$ and $b$ to $d$ through the crosscap. (See Figure~\ref{fig:cross-cap}.)
The
graph we seek is obtained
from the original $\overline{H_2}$ by
adding the new crosscap and the above seven new edges.
This graph, then,
has (odd) non-orientable genus $\overline g +1$.  \medskip

The above examples show that $k_G(x,y)=2-\chi(\Sigma)$ does not guarantee the existence of
a $W_v$-path joining $x$ and $y$. Therefore  $f(\Sigma)\ge 3-\chi(\Sigma)$ for surfaces $\Sigma$
with $\chi(\Sigma)< -1$.
By Corollary~\ref{cor:main}, Theorem~\ref{thm:main-general} follows.\medskip

\section{Polyhedral maps on the projective plane}

In this section, we obtain a sharp lower bound for the number of internally disjoint non-revisiting
$(x,y)$-paths for graphs polyhedrally embedded in the projective plane which improves
the bound given in Theorem~\ref{thm:main2}. Barnette's result \cite{B2} for the projective plane is
a direct corollary of this result.

In the following, two closed curves $\alpha$ and $\beta$ are {\em homotopically disjoint} if there exist two disjoint closed
curves $\alpha'$ and $\beta'$ such that $\alpha$ is homotopic to $\alpha'$ and $\beta$ is homotopic
to $\beta'$.

\begin{thm}\label{thm:proj}
Let $G$ be a graph polyhedrally embedded in the projective plane and suppose $x$ and $y$ are two non-cofacial vertices.
Then there are at least $\kappa_G(x,y)-2$ internally disjoint $W_v$-paths joining $x$ and $y$.
\end{thm}

\begin{proof} Let ${\cal P}=\{P_1,\ldots,P_k\}$
be a family of internally disjoint $(x,y)$-paths such that
$r_{\cal P}=\sum^k_{i=1} r_{P_i}$
is minimum.
If $r_{\cal P}=0$, we are done, so in the following we will assume that $r_{\cal P}> 0$.

Suppose $P_1\in {\cal P}$.
Define ${\cal P}_A$ by ${\cal P}_A=\{P_i\in {\cal P}\vert P_i \ {\rm is\ homotopic\ to\ } P_1\}$.
Trivially, $P_1\in {\cal P}_A$.
Now define ${\cal P}_B$ by ${\cal P}_B={\cal P}-{\cal P}_A$.
Then for any $P_i,P_j\in {\cal P}_A$,
$P_i\cup P_j$ bounds a disk.
Moreover, if $P_i\in {\cal P}_A$ and $P_{\alpha}\in {\cal P}_B$, $P_i\cup P_{\alpha}$ is a non-contractible cycle since
$P_i$ and $P_{\alpha}$ are not homotopic.
Note that there is only one homotopy class of non-contractible simple closed curves on the projective plane since
the fundamental group of this surface is
$\mathbb Z _2$.
So all non-contractible cycles of $G$ are homotopic.
For $P_{\alpha},P_{\beta}\in {\cal P}_B$,
$P_1\cup P_{\alpha}$ is homotopic to $P_1\cup P_{\beta}$.
It follows that $P_{\alpha}$ is homotopic to $P_{\beta}$.
Hence ${\cal P}_B$ is also a homotopy class of internally disjoint $(x,y)$-paths.
\par

Without loss of generality, we may write ${\cal P}_A=\{P_1,\ldots,P_t\}$ and ${\cal P}_B=\{P_{t+1},\ldots,P_k\}$,
and also without loss of generality, we may assume that $|{\cal P}_A|\ge |{\cal P}_B|$.
Note that $k\ge 3$ since $G$ is 3-connected and ${\cal P}_B$ may be empty.
In any case
$t\ge 2$.

Since ${\cal P}_A$ is a homotopy class, $P_i\cup P_j$ bounds a disk, for any two distinct $P_i,P_j\in {\cal P}_A$.
Therefore, all paths in ${\cal P}_A$ are contained in a closed disk $D$ bounded by the union of two paths in this set.
Without loss of generality, let us renumber the paths if necessary, so that these two paths are denoted by $P_1$ and $P_t$. (See 
Figure~\ref{fig:2-class} where the disk $D$ is represented by the shaded region.)
Similarly, we may suppose that paths $P_{t+1}$ and $P_k$ bound a closed disk $D'$ containing all the paths in ${\cal P}_B$.

\begin{figure}[!hbtp] \refstepcounter{figure}\label{fig:2-class}
\begin{center}
\includegraphics[scale=.8]{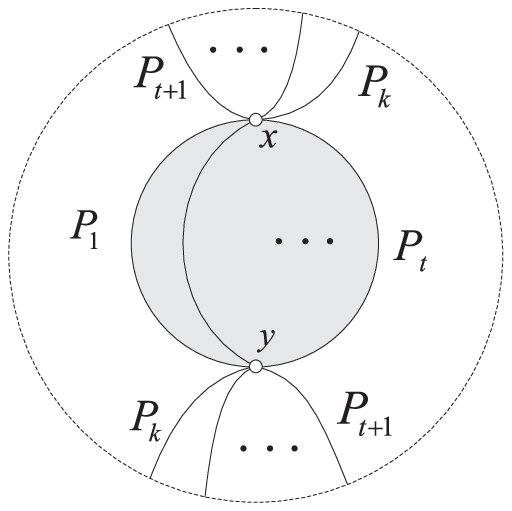}\\
{Figure~\ref{fig:2-class}:  The two homotopy classes ${\cal P}_A$ and ${\cal P}_B$ of all $(x,y)$-paths.}
\end{center}
\end{figure}

By Lemma~\ref{lem:contractible} and the minimality of $r_{\cal P}$, every revisit to any face $F$ by an $(x,y)$-path
in $\cal P$ is non-contractible.
So every face contained in the disk $D$ bounded by $P_1\cup P_t$ (respectively, in the disk $D'$ bounded by $P_{t+1}\cup P_k$)
is not revisited by any path in ${\cal P}$.
Hence a face can only be revisited by $P_1,P_t,P_{t+1}$ or $P_k$.\par

If ${\cal P}_B=\emptyset$, then a face can be revisited only by $P_1$ or $P_t$ ($= P_k$), so in this instance,
there are at least $\kappa_G(x,y)-2$ $W_v$-paths.
So assume that ${\cal P}_B\ne\emptyset$.
Let $F$ be a face revisited by path $P_1$.
By Lemma~\ref{lem:non-contractible}, $F\cap P_1$ has exactly two components $S_1$ and $S_2$.

\medskip

\noindent
{\bf Claim 1:} {\sl One of $S_1$ and $S_2$ is the single vertex $x$ or $y$.}\par
\medskip
\noindent
{\it Proof of Claim 1.}  Suppose to the contrary that $S_1-\{x,y\}$ contains a vertex $u$ and that $S_2-\{x,y\}$ contains a vertex $v$.
The dual closed curve $\ell_{12}$ of $\{S_1,S_2\}$ through $u$ and $v$ does not intersect $P_t\cup P_k$ which is a non-contractible cycle.
Therefore, $\ell_{12}$ is contractible.
Hence $\{S_1,S_2\}$ is a contractible revisit, a contradiction.
This completes the proof of Claim 1.
\medskip

If both homotopy classes ${\cal P}_A$ and ${\cal P}_B$ contain at most one path that is not a $W_v$-path, then trivially there are at least
$\kappa_G(x,y)-2$ $W_v$-paths. So in the following we will assume, without loss of generality, that class ${\cal P}_A$ contains exactly two paths that are not $W_v$-paths, $P_1$ and $P_t$, since a face of $G$ can only be revisited by $P_1, P_t, P_{t+1}$ or $P_k$.

\medskip

\noindent
{\bf Claim 2:}  {\sl The paths $P_1$ and $P_t$ cannot revisit the same face.}\par
\medskip

\noindent
{\it Proof of Claim 2:}  Suppose to the contrary that there exists a face $F$ which is revisited by paths
$P_1$ and $P_t$.
By Lemma~\ref{lem:non-contractible}, $F\cap P_1$ has two components $S_1$ and $S_2$.
By Claim 1, one of $S_1$ and $S_2$ is the single vertex $x$ or $y$.
Suppose without loss of generality that $S_1=\{x\}$.
Similarly, $F\cap P_t$ has two components and one of them is the single vertex $x$ or $y$.
Since $x$ and $y$ are not cofacial, the vertex $y$ cannot be a single vertex component of 
$F\cap P_t$.
Therefore, $S_1=\{x\}$ is also a component of $F\cap P_t$.
Let $S_3$ be the other component of $F\cap P_t$.

Let $\ell_{12}$ and $\ell_{13}$ be the dual closed curves of $\{S_1,S_2\}$ and $\{S_1,S_3\}$ respectively.
Note that both $\ell_{12}$ and $\ell_{13}$ are non-contractible.
Therefore, $\ell_{12}$ and $\ell_{13}$ cross transversally at the vertex $x$.
By the definition of dual closed curves, we assume that $\ell_{12}$ and $\ell_{13}$
intersect only at $x$
(otherwise, other intersection components lie either in the face $F$ or $N_{\tp}(P_1)\cap N_{\tp}(P_t)$, and hence can be contracted to $x$).
Let $D''=D\cup N_{\tp}(P_1)\cup N_{\tp}(P_t)$. Then the face $F$ touches the disk $D'$ four times along $\ell_{12}$ and $\ell_{13}$ at $S_1=\{x\}, S_2$ and $S_3$.
So the boundary of $F$ self-intersects at $x$ which contradicts the fact that $G$ is polyhedrally embedded in the projective plane.
This completes the proof of Claim 2.\medskip

By Claim 2, $P_1$ and $P_t$ revisit two distinct faces $F_1$ and $F_2$.
By Lemma~\ref{lem:non-contractible}, $P_1\cap F_1$ has exactly two components $S^1_1$ and $S^1_2$ and
$P_t\cap F_2$ has exactly two components $S^t_1$ and $S^t_2$. 
Next we 
show that both $P_{t+1}$ and $P_k$ are $W_v$-paths.

Assume there is a face $F$ revisited by a path from ${\cal P}_B$, say $P_k$.
Note that the boundary of $F$ is homotopically disjoint  from the boundary of $D\cup F_1\cup F_2$, and therefore,
the boundary of $D'\cup F$ is homotopically disjoint from the boundary of $D\cup F_1\cup F_2$.
Let $\ell_{12}$ be the dual closed curve of the revisits $\{S_1^1, S_2^1\}$ of $F_1$ by $P_1$ and $\ell'$ be the dual curve of the revisits
of $F$ by $P_k$.
Therefore, $\ell_{12}$ and $\ell'$ are homotopically disjoint, a
contradiction to the fact that both $\ell_{12}$ and $\ell'$ are non-contractible. This contradiction implies that
$P_k$ is a $W_v$-path.
Similarly,
so is $P_{t+1}$.
It follows then that $G$ contains at least $\kappa_G(x,y)-2$ internally disjoint $W_v$-paths.
\end{proof}

\begin{figure}[!hbtp] \refstepcounter{figure}\label{fig:3-conn}
\begin{center}
\includegraphics[scale=1]{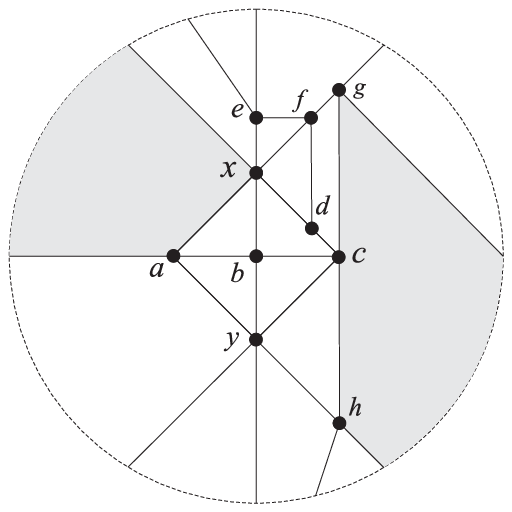}\\
{Figure~\ref{fig:3-conn}: An example.}
\end{center}
\end{figure}

\noindent{\bf Remark:}  The
example
shown in Figure~\ref{fig:3-conn}
shows that the bound of $\kappa_G(x,y)-2$ in Theorem~\ref{thm:proj} for the projective plane is best possible.
In the graph shown
in this figure, there are six internally disjoint $(x,y)$-paths:  $xay$, $xby,$
$xdcy$, $xey$, $xfgy$ and $xhy$.
Hence $\kappa_G(x,y)=6$.
But there are only four internally disjoint non-revisiting $(x,y)$-paths:  $xay,xby,xey$ and $xhy$ as both $xdcy$ and
$xfgy$ revisit the
(shaded)
face bounded by $axhcga$.

\section{Concluding remarks}
Let $\Sigma$ be a closed surface and $G$ be a graph polyhedrally embedded in $\Sigma$. A result of Cook \cite{Co} shows that the connectivity of $G$ is at most $(5+\sqrt{49-24\chi(\Sigma)})/2$ if  $\chi(\Sigma)\le 0$. It then follows that, if $\chi(\Sigma)<-7$, the connectivity of $G$ 
is less than $3-\chi(\Sigma)$. Then $G$ may not have a $W_v$-path for some pair of vertices by Theorem~\ref{thm:main-general}. However, the locally connectivity of a pair of vertices $x$ and $y$ of $G$ could be arbitrarily large. Hence, in the definition of $f(\Sigma)$, the local connectivity cannot be replaced by the connectivity of $G$. 

Theorem~\ref{thm:main-general} shows linear bounds for $f(\Sigma)$ for surfaces $\Sigma$. The previous results of Barnette \cite{B1,B2,B3} and Pulapaka and Vince \cite{PV2} show that $f(\Sigma)=3$ for surfaces with $\chi(\Sigma)\ge 0$. However, the exact values $f(\Sigma)$ for surfaces $\Sigma$ with $\chi(\Sigma)<0$ are unknown. It is interesting to ask the following question.

\begin{prob}
Let $\Sigma$ be a closed surface with $\chi(\Sigma)\le -1$. Determine the exact value of $f(\Sigma)$. 
\end{prob}

A solution to the above question would settle the existence problem of $W_v$-path in graphs polyhedrally embedded in the surface $\Sigma$ with $\chi(\Sigma)=-1$, the only surface for which the existence of a $W_v$-path between a pair of vertices of $G$ remains unknown.

Theorem~\ref{thm:main2} provides a lower bound for the number of internally disjoint $W_v$-paths between a pair of vertices of $G$. The bound is sharp for the sphere, but may not be sharp for other surfaces. Indeed, it is not tight for the projective plane.  We propose the following.

\begin{prob}\label{prob:2}
Let $G$ be a graph polyhedrally embedded in the surface $\Sigma$ and let $x$ and $y$ be two non-cofacial vertices. Find a sharp lower bound for the number of internally disjoint non-revisiting $(x,y)$-paths. 
\end{prob}

Theorem~\ref{thm:proj} evidences that the number of internally disjoint non-revisiting $(x,y)$-paths is related to the Euler characteristic of the surface. But the connection is not clear. A solution to Problem~\ref{prob:2} for the torus or the Klein bottle is interesting, which may lead to a complete solution to the problem.  

\bigskip

\noindent{\bf Acknowledgement.} The authors would like to thank  the anonymous referees for their valuable comments which improved the final version of the paper.

\end{document}